\documentclass[reqno, 11pt]{amsart}
\usepackage{amssymb}
\usepackage[numbers]{natbib}

\usepackage[nesting]{hyperref}

\usepackage[pdftex]{graphicx}
\usepackage{listings}
\usepackage{multirow}
\usepackage{placeins}
\usepackage{color}
\usepackage{subfigure}
\usepackage{lscape}

\newcommand{\BlackBoxes}{\global\overfullrule5pt}

\BlackBoxes

\newcommand{\R}{\mathbb{R}}
\newcommand{\N}{\mathbb{N}}

\newcommand{\E}{\mathbb{E}}

\newcommand{\PP}{\mathrm{P}}
\newcommand{\Eop}{\operatorname{\mathbb{\E}}}

\newcommand{\Pop}{\operatorname{\mathbb{\PP}}}

\newtheorem{theorem}{Theorem}

\newtheorem{lemma}[theorem]{Lemma}

\theoremstyle{definition}
\newtheorem{example}[theorem]{Example}
\newtheorem{remark}[theorem]{Remark}

\numberwithin{equation}{section} \numberwithin{theorem}{section}

\def\0{\kern0pt\-\nobreak\hskip0pt\relax}

\makeatletter
\AtBeginDocument{%
 \def\@serieslogo{%
 \vbox to\headheight{%
 \parindent\z@ \fontsize{6}{7\p@}\selectfont
 \vss}}}

\def\makeoverbar#1#2#3#4#5#6#7{%
 \setbox0=\hbox{$\m@th#2\mkern#5mu{{}#3{}}\mkern#6mu$}%
 \setbox1=\null \dimen@=#4\fontdimen8#13 \dimen@=3.5\dimen@
 \advance\dimen@ by \ht0 \dimen@=-#7\dimen@ \advance\dimen@ by \wd0
 \ht1=\ht0 \dp1=\dp0 \wd1=\dimen@
 \dimen@=\fontdimen8#13 \fontdimen8#13=#4\fontdimen8#13
 \rlap{\hbox to \wd0{$\m@th\hss#2{\overline{\box1}}\mkern#5mu$}}
 \fontdimen8#13=\dimen@}

\def\mylabel#1#2{{\def\@currentlabel{#2}\label{#1}}}

\makeatother

\begin{document}


\makeatletter \providecommand\@dotsep{5} \makeatother

\title[Zero-sum Risk-Sensitive Stochastic Games]{Zero-sum Risk-Sensitive Stochastic Games}

\author[N. \smash{B\"auerle}]{Nicole B\"auerle${}^*$}
\address[N. B\"auerle]{Institute for Stochastics,
Karlsruhe Institute of Technology, D-76128 Karlsruhe, Germany}

\email{\href{mailto:nicole.baeuerle@kit.edu}
{nicole.baeuerle@kit.edu}}


\author[U. \smash{Rieder}]{Ulrich Rieder$^\ddag$}
\address[U. Rieder]{University of Ulm, D-89069 Ulm, Germany}

\email{\href{mailto:ulrich.rieder@uni-ulm.de} {ulrich.rieder@uni-ulm.de}}

\maketitle

\begin{abstract}
In this paper we consider two-person zero-sum risk-sensitive stochastic dynamic games with Borel state and action spaces and bounded reward. The term {\em risk-sensitive} refers to the fact that instead of the usual risk neutral optimization criterion we consider the exponential certainty equivalent. The discounted reward case on a finite and an infinite time horizon is considered, as well as the ergodic reward case. Under continuity and compactness conditions we prove that the value of the game exists and solves the Shapley equation and we show the existence of optimal (non-stationary) strategies. In the ergodic reward case we work with a local minorization property and a Lyapunov condition and show that the value of the game solves the Poisson equation. Moreover, we prove the existence of optimal stationary strategies. A simple example highlights the influence of the risk-sensitivity parameter. Our results generalize findings in \cite{bg14} and answer an open question posed there.
\end{abstract}

\vspace{0.5cm}
\begin{minipage}{14cm}
{\small
\begin{description}
\item[\rm \textsc{ Key words}]
{\small Risk-sensitive stochastic games, ergodic reward, Shapley equation, Poisson equation.}
\end{description}
}
\end{minipage}

\section{Introduction}\label{sec:intro}\noindent
In this paper we consider two-person zero-sum risk-sensitive stochastic dynamic games with finite or infinite time horizon. The discounted reward case as well as the ergodic reward case are considered. The term {\em risk-sensitive} refers to the fact that instead of the usual risk neutral optimization criterion we consider the exponential certainty equivalent which for a random variable $X$ and a risk sensitivity parameter $\gamma\neq 0$ is defined by
$$ \rho(X) = \frac 1\gamma\ln \Big( \Eop e^{\gamma X}\Big).$$
For small $\gamma$ a Taylor series expansion reveals that
$$ \rho(X) = \Eop X +\frac12 \gamma Var(X) +O(\gamma^2)$$
and that $\gamma\to 0$ corresponds to the risk-neutral case. Hence the exponential certainty equivalent constitutes a risk-adjusted expectation similar to the classical mean-variance criterion. For deeper economical insight, we refer the reader to \cite{bpli03}.  In case $\gamma >0$, the variance is added and the first player is risk-seeking whereas for $\gamma<0$ the variance is subtracted and the first player is risk-averse. In the context of a zero-sum game with $\gamma<0$ where the first player  wants to maximize the  exponential certainty equivalent  and the second player wants to minimize the  exponential certainty equivalent, the second player is often interpreted as the {\em nature} which works against the first player. Thus the game can be seen as a kind of worst-case optimization. Risk-sensitive optimization problems have been considered since the seminal paper \cite{hm72}, but it was only lately that this topic gained renewed interest. This is mainly due to applications in finance.

Whereas there are a lot of papers on risk-sensitive Markov Decision processes under different optimization criteria (see e.g. \cite{hm72,whi90,jjr94,hhm99,DiMasiStettner,bm02,cchh05,ds07,j07,cchh11,br14,sos15,bj14,br15}), there are only a few papers on risk-sensitive stochastic games in discrete time. In \cite{klo00} the author considers Nash equilibria for a two-person non-zero-sum game with a quadratic-exponential cost criterion and in \cite{jn14} the authors treat so-called overlapping generations models. In \cite{bg14} the authors consider two-person zero-sum risk-sensitive stochastic games with countable state space under the discounted cost criterion and the ergodic cost criterion. In their model the authors replace the one-stage cost in the exponential by the average over the randomized policies and later make the simplifying assumption that the one-stage cost does not depend on the actions of the two players. The average  cost problem is solved with a standard uniform ergodicity condition.
There are more papers on risk-sensitive stochastic differential games, i.e. games in continuous time (see e.g. \cite{ekh03,fhh11,bg12,dj15}). But the solution techniques via Backward Stochastic Differential Equations and Hamilton-Jacobi-Bellman-Isaacs equations are quite different.

In this paper we extend the results of \cite{bg14}. Throughout we work with Borel state and action spaces and general one-stage rewards which may depend on actions of both players. The one-stage rewards are assumed to be bounded. Since we use randomized actions, we interpret expectations w.r.t.\ the evolution of the state process and the randomization over actions. First we consider the finite horizon stochastic game, then the infinite horizon discounted game. We show that the value of the game exists and satisfies the Shapley equation. Moreover, we prove the existence of optimal strategies for both players which has been posed as an open problem in \cite{bg14}.  In order to be able to apply the usual minimax theorem for the existence of saddle points, we need convexity of the expression in the Shapley equation w.r.t.\ to the randomization measures for the actions of player 1 and 2. This is achieved by applying a trick which has been used in \cite{jjr94} before: We consider a modified transition measure which is not necessarily a probability measure any more and contains part of the reward. The same approach also works for the infinite horizon with discount factor $\beta<1$. We again obtain the Shapley equation for the value of the game and show the existence of optimal strategies - a statement which has been posed as an open question in \cite{bg14} (Remark 1). We treat both the risk-seeking ($\gamma >0$) and the risk-averse case ($\gamma <0$). In the ergodic reward case we assume $\beta=1$ and use a local minorization property together with a Lyapunov condition in order to obtain ergodicity (see \cite{HM,sos15}). For $\gamma=0$, i.e.\ in the risk-neutral situation, our condition reduces to the one in \cite{HM} for the ergodicity of Markov chains.  Here we show for small $\gamma$ that the value of the game does not depend on the initial state and is a partial solution of the Poisson equation. Moreover, there exist optimal stationary strategies for both players.

Our paper is organized as follows: In the next section we introduce our model and formulate the problem for a finite time horizon. Moreover, we introduce some continuity and compactness conditions which we use throughout our paper for the existence of optimal strategies. In section \ref{sec:fh} we first consider the problem with a finite time horizon and prove the Shapley equation for the value of the game and the existence of optimal (non-stationary) strategies. In this section we also present a simple example which highlights the influence of the risk sensitivity parameter $\gamma$ on the optimal strategies for the players. In Section \ref{sec:infh} we consider the infinite horizon discounted reward game. Under the same condition as in Section \ref{sec:fh} we show a Shapley equation for the optimal value of the game and prove the existence of optimal stationary strategies. Finally in Section \ref{sec:arg} we solve the ergodic risk-sensitive game under local ergodicity conditions.

\section{Model Formulation}\label{sec:intro}\noindent
We suppose that the following {\em two-person zero-sum} stochastic game is given: The state space is a Borel set $\mathbb{E}$, the action spaces for player 1 and 2 are denoted by $\mathbb{A}$ and $\mathbb{B}$ respectively and are Borel spaces too.
It is assumed that $\mathbb{E}\times\mathbb{A}$ contains the graph of a measurable map from $\mathbb{E}$ into $\mathbb{A}$ and that  $\mathbb{E}\times\mathbb{B}$ contains the graph of a measurable map from $\mathbb{E}$ into $\mathbb{B}$. For any $x\in \mathbb{E}$, the non-empty and measurable $x$-section $\mathbb{A}_x$ of $\mathbb{A}$ denotes the set of all {\em admissible actions} for player 1, if the system is in state $x$. The analogous assumptions and notations are used for player 2. We suppose that there is a {\em transition kernel} $Q$ from $\mathbb{E}\times\mathbb{A}\times\mathbb{B}$ to $\mathbb{E}$ and
a measurable reward function $r: \mathbb{E}\times\mathbb{A} \times\mathbb{B}\to \R$ for player 1 which is {\em bounded} and without loss of generality $0\le r \le\bar{r}$.

The game is played as follows: At each stage both players observe the current state $x\in\mathbb{E}$ and choose actions $a\in \mathbb{A}_x$ and $b\in \mathbb{B}_x$ independently of each other. Player 1 then receives the reward $r(x,a,b)$. Afterwards the system moves to a new state according to the transition kernel $Q(\cdot|x,a,b)$. The reward may be discounted by a factor $\beta\in(0,1]$. The aim of the first player is to maximize the risk-sensitive accumulated reward of the stochastic game over a finite or an infinite time horizon and the aim of the second player is to minimize it. Before we introduce the objectives, let us first define the policies. Since even for ordinary risk-neutral games it is known that optimal policies can only be found in the class of randomized policies, we have to consider this class too. Let us denote by $\mathbb{P}(\mathbb{A})$ the set of all probability measures on $\mathbb{A}$, endowed with the weak topology. Then
$$ \{ (x,\mu,\nu) : x\in \mathbb{E}, \mu \in \mathbb{P}(\mathbb{A}_x), \nu \in  \mathbb{P}(\mathbb{B}_x)\}$$
is a measurable subset of $ \mathbb{E}\times \mathbb{P}(\mathbb{A})\times \mathbb{P}(\mathbb{B})$.
We will denote by $\mathbb{F}$ the set of all measurable mappings $f:\mathbb{E}\to \mathbb{P}(\mathbb{A})$ such that $f(x) \in \mathbb{P}(\mathbb{A}_x)$ for all $x\in\mathbb{E}$. The decision rule $f$ is interpreted as a randomized decision rule for player 1. A {\em randomized Markovian policy} is then a sequence of decision rules $(f_n)$ with $f_n\in \mathbb{F}$. Similarly for player 2, we denote the set of all randomized decision rules by $\mathbb{G}$ and elements by $g\in \mathbb{G}$. In this paper we do not consider non-Markovian policies, since as for the risk-neutral game, it can be shown that the value functions remain unchanged if the set of admissible policies is enlarged to the class of history-dependent policies. This has first been shown for risk-neutral Markov Decision Processes in Theorem 18.4 of \cite{hin70}. For general certainty equivalents the question has been considered in \cite{br14}. Though in general optimal policies are history dependent, the statement holds for the exponential certainty equivalent. The proof given in \cite{br14} directly generalizes to the game setting. A randomized {\em stationary} policy is a Markovian policy $(f_n)$ where $f_n=f$ independent of $n$.

Two randomized policies $\pi=(f_n)$ and $\sigma=(g_n)$ together with the initial state $x$ define according to the Theorem of Ionescu Tulcea a probability measure $\mathrm{P}_x^{\pi\sigma}$ on $(\mathbb{E}\times \mathbb{A}\times \mathbb{B})^\infty$ by
\begin{eqnarray*} && \Pop_x^{\pi\sigma} (A_0\times B_0\times C_1\times \ldots A_{n-1}\times B_{n-1}\times C_n\times\mathbb{E}\times \mathbb{A}\times \mathbb{B}\times\ldots ) = \\
&=&\int_{A_0}\int_{B_0}\int_{C_1}\ldots \int_{C_n}Q(dx_n|x_{n-1},a_{n-1},b_{n-1})f_{n-1}(x_{n-1})(da_{n-1}) g_{n-1}(x_{n-1})(db_{n-1})\ldots\\
&& \ldots Q(dx_1|x_0,a_0,b_0) f_{0}(x_{0})(da_{0}) g_{0}(x_{0})(db_{0}). \end{eqnarray*}
The reward generated by the policies over $N$ stages is
\begin{equation*} R_N := \sum_{k=0}^{N-1} \beta^k r(X_k,A_k,B_k).
\end{equation*}
The risk-sensitive reward (exponential certainty equivalent) for player 1 under the policies $\pi$ and $\sigma$ is then given by
\begin{equation}\label{eq:CE} \frac1\gamma \ln \Eop_x^{\pi\sigma}\Big[ \exp\big( \gamma R_N\big) \Big]\end{equation}
where $\Eop_x^{\pi\sigma}$ is the expectation under $\Pop_x^{\pi\sigma}$ and $(A_0,B_0,X_1,\ldots)$ is the process under $\pi$ and $\sigma$. We first assume that the risk-sensitivity parameter $\gamma\in (0,\bar{\gamma}]$.
Since $\gamma$ is positive and the logarithm is monotone we can equivalently consider the optimization criterion
\begin{equation} V_{N\pi\sigma}(x,\gamma) := \Eop_x^{\pi\sigma}\Big[ \exp\big( \gamma R_N\big) \Big].\end{equation}
The term zero-sum game refers to the fact that the total utility of player 1 is $\exp\big( \gamma R_N \big)$ and we assume that the total utility of player 2 is $-\exp\big( \gamma R_N\big)$.
The corresponding {\em upper value} of the game over $N$ stages is given by
\begin{equation}\label{eq:uppervalue} V_{N}(x,\gamma) := \inf_\sigma \sup_\pi V_{N\pi\sigma}(x,\gamma). \end{equation}
If $\inf_\sigma \sup_\pi V_{N\pi\sigma}(x,\gamma) = \sup_\pi  \inf_\sigma V_{N\pi\sigma}(x,\gamma)$, the function $V_N$ is called {\em value function} of the stochastic game. A policy $\pi^*$ is called {\em optimal for player 1} for the $N$-stage game if
\begin{equation*} V_{N\pi^*\sigma}(x,\gamma) \ge \sup_\pi \inf_{\sigma'}  V_{N\pi\sigma'}(x,\gamma),\; \mbox{ for all } \sigma. \end{equation*}
 A policy $\sigma^*$ is called {\em optimal for player 2} for the $N$-stage game if
\begin{equation*} V_{N\pi\sigma^*}(x,\gamma) \le  \inf_\sigma \sup_{\pi'} V_{N\pi'\sigma}(x,\gamma),\; \mbox{ for all } \pi. \end{equation*}
A pair of policies $(\pi^*,\sigma^*)$ for which
\begin{equation*}\label{eq:saddlepoint}
    V_{N\pi\sigma^*} \le V_{N\pi^*\sigma^*} \le V_{N\pi^*\sigma},\quad\mbox{for all } \pi,\sigma
\end{equation*}
is called {\em saddle-point equilibrium} and implies that $\pi^*$ is optimal for player 1, $\sigma^*$ is optimal for player 2 and that $V_{N\pi^*\sigma^*} =V_N$.
In order to ease notation we will in general call $(\pi^*,\sigma^*)$ a saddle-point of a function $v(\pi,\sigma)$ if
\begin{equation*}
    v(\pi,\sigma^*) \le v(\pi^*,\sigma^*) \le v(\pi^*,\sigma),\quad\mbox{for all } \pi,\sigma.
\end{equation*}
Throughout we make the following assumptions (see e.g. \cite{br11}):

\begin{enumerate}
  \item[(A1)] The sets $\mathbb{A}_x$ and $\mathbb{B}_x$ are compact for every $x\in\mathbb{E}$.
  \item[(A2)] The correspondences $x\mapsto \mathbb{A}_x$ and $x\mapsto \mathbb{B}_x$ are continuous. We call a correspondence $x\mapsto \mathbb{A}_x$ continuous if it is upper-semicontinuous, meaning that $\{x\in \mathbb{E} : \mathbb{A}_x\cap C\neq \emptyset\}$ is closed in $\mathbb{E}$ for every closed subset $C\subset \mathbb{A}$, and lower-semicontinuous if $\{x\in \mathbb{E} : \mathbb{A}_x\cap D\neq \emptyset\}$ is open in $\mathbb{E}$ for every open subset $D\subset \mathbb{A}$.
  \item[(A3)] The mapping $(x,a,b) \mapsto r(x,a,b)$ is continuous.
  \item[(A4)] The transition measure $Q$ is weakly continuous, i.e. for all bounded and continuous $v:\mathbb{E}\to\R$ we have that $(x,a,b) \mapsto \int v(x') Q(dx'|x,a,b)$ is continuous.
  \end{enumerate}

Note that the assumptions are of course satisfied when $\mathbb{E}$ is countable and $\mathbb{A}_x$ and $\mathbb{B}_x$ are finite.

\section{Finite Horizon Discounted Game}\label{sec:fh}
We solve the problem by transforming it to the classical risk-neutral setup with a different transition measure. Such a transformation has also been used in the context of risk-sensitive partially observable models (see \cite{jjr94}). Recall that we first restrict to the case $\gamma>0$. Let us define a transition measure $\tilde{Q}$ from $\mathbb{E}\times(0,\bar{\gamma}]\times\mathbb{A}\times\mathbb{B}$ to $\mathbb{E}$ by
\begin{equation}\label{eq:tildeQ}
    \tilde{Q}(dx'|x,\gamma,a,b) := e^{\gamma r(x,a,b)} Q(dx'|x,a,b), \quad (x,\gamma,a,b) \in \mathbb{E}\times(0,\bar{\gamma}]\times\mathbb{A}\times\mathbb{B}.
\end{equation}
It follows directly from this definition that $\tilde{Q}$ is also weakly continuous in $(x,\gamma,a,b)$. Moreover, define for $\mu\in\mathbb{P}(\mathbb{A}_x)$ and for $\nu\in \mathbb{P}(\mathbb{B}_x)$
\begin{equation}\label{eq:tildeQ3}
    \tilde{Q}(dx'|x,\gamma,\mu,\nu) := \int\int \tilde{Q}(dx'|x,\gamma,a,b)\mu(da)\nu(db), \quad (x,\gamma) \in \mathbb{E}\times (0,\bar{\gamma}].
\end{equation}

For $f\in \mathbb{F}, g\in \mathbb{G}$, $\mu\in\mathbb{P}(\mathbb{A}_x), \nu\in \mathbb{P}(\mathbb{B}_x)$ and a bounded, measurable function $v:\mathbb{E}\times (0,\bar{\gamma}]\to\R$ let us introduce the following operators:
\begin{eqnarray*}
  (T_{fg}v)(x,\gamma) &:=& \int v(x',\beta \gamma) \tilde{Q}(dx'|x,\gamma,f(x),g(x)),\\
  (Lv)(x,\gamma,\mu,\nu) &:=& \int v(x',\beta \gamma) \tilde{Q}(dx'|x,\gamma,\mu,\nu), \\
  (Tv)(x,\gamma) &:=&  \inf_g \sup_f (T_{fg}v)(x,\gamma)\\
  &=& \inf_\nu\sup_\mu (Lv)(x,\gamma,\mu,\nu).
\end{eqnarray*}



In what follows we always set $V_0(x,\gamma) = 1$. The value of the game for fixed policies can be computed with the help of the $T_{fg}$-operator.

\begin{theorem}\label{theo:iteration}
Let $\pi=(f_0,f_1,\ldots)$ and $\sigma=(g_0,g_1,\ldots)$ be randomized policies for player 1 and 2 respectively. Then it holds for all $n=1,\ldots ,N$ that
\begin{itemize}
  \item[a)] $ V_{n\pi\sigma}(x,\gamma)= \int V_{n-1 \overrightarrow{\pi},\overrightarrow{\sigma}}(x',\beta \gamma) \tilde{Q}\big(dx'|x,f_0(x),g_0(x),\gamma\big) = (T_{f_0g_0}V_{n-1 \overrightarrow{\pi},\overrightarrow{\sigma}})(x,\gamma),$ where $\overrightarrow{\pi}=(f_1,f_2,\ldots)$ and $\overrightarrow{\sigma}=(g_1,g_2,\ldots)$ are the shifted policies.
  \item[b)] $ V_{n\pi\sigma}= T_{f_0g_0}T_{f_1g_1}\ldots T_{f_{n-1}g_{n-1}} V_0.$
\end{itemize}

\end{theorem}

\begin{proof}
It suffices to prove part a), since part b) follows directly from a) and the definition of the operators.
Part a) is done by induction. For $n=1$ we obtain
\begin{eqnarray*}
  V_{1\pi\sigma}(x,\gamma) &=&  \Eop_x^{\pi\sigma} \Big[ \exp\big(\gamma r(X_0,A_0,B_0)\big)\Big] \\
   &=&  \int\int \exp\big( \gamma r(x,a,b)\big) f_0(x)(da) g_0(x)(db)  \\
   &=& \int V_{0}\; \tilde{Q}\big(dx'|x,\gamma,f_0(x),g_0(x)\big).
\end{eqnarray*}
Now suppose the statement is true for $n-1$. We obtain
\begin{eqnarray*}
  && V_{n\pi\sigma}(x,\gamma) =  \Eop_x^{\pi\sigma}\Big[ \exp\Big( \gamma \sum_{k=0}^{n-1} \beta^k r(X_k,A_k,B_k)\Big) \Big] \\
   &=& \int\int\int \exp\big(\gamma r(x,a,b)\big) \cdot\\
   && \cdot\Eop_{x_1}^{ \overrightarrow{\pi}\overrightarrow{\sigma}}\Big[\exp\Big( \gamma \sum_{k=1}^{n-1} \beta^k r(X_k,A_k,B_k)\Big)\Big] Q(dx_1| x,a,b) f_0(x)(da) g_0(x)(db) \\
   &=& \int \Eop_{x_1}^{ \overrightarrow{\pi}\overrightarrow{\sigma}}\Big[\exp\Big( \beta\gamma \sum_{k=0}^{n-2} \beta^k r(X_k,A_k,B_k)\Big)\Big] \tilde{Q}(dx_1| x,\gamma,f_0(x),g_0(x))
\end{eqnarray*}
and the induction hypothesis implies the statement.
\end{proof}

The next theorem provides the solution of the problem. Let us denote by $\mathcal{C}$ the set of continuous and bounded functions on $\mathbb{E}\times (0,\bar{\gamma}]$.

\begin{theorem}\label{theo:finite}
Assume (A1)-(A4).
\begin{itemize}
  \item[a)] For all $n=1,\ldots,N$ it holds that $V_n\in \mathcal{C}$ and $$V_n(x,\gamma) = (TV_{n-1})(x,\gamma) = \sup_\mu \inf_\nu \int V_{n-1}(x',\beta \gamma) \tilde{Q}(dx'|x,\gamma,\mu,\nu).$$
  \item[b)]  For $n=1,\ldots,N$ there exist measurable functions  $(\mu_n^*,\nu_n^*):\mathbb{E}\times (0,\bar{\gamma}]\to \mathbb{P}(\mathbb{A})\times \mathbb{P}(\mathbb{B}) $ which are admissible, i.e. $\mu_n^*(x,\gamma)\in \mathbb{P}(\mathbb{A}_x)$ and $\nu_n^*(x,\gamma)\in \mathbb{P}(\mathbb{B}_x)$ such that $(\mu_n^*(x,\gamma),\nu^*(x,\gamma))$ is a saddle point of
      $$ (\mu,\nu) \mapsto  LV_{n-1}\big(x,\gamma,\mu,\nu\big) $$ for all $(\mu,\nu) \in  \mathbb{P}(\mathbb{A}_x)\times\mathbb{P}(\mathbb{B}_x)$.
    Then $V_N$ is the value of the $N$-stage stochastic game and $(\pi^*,\sigma^*)=({f}_n^*,{g}_n^*)_{n=0,\ldots,N-1}$ with ${f}_n^*(x) := \mu_{N-n}^*(x,\gamma\beta^n)$ and  ${g}_{n}^*(x) := \nu_{N-n}^*(x,\gamma\beta^n)$ are optimal policies for player 1 and 2 respectively.
\end{itemize}
\end{theorem}

\begin{proof}
First we define for all admissible $(\mu,\nu): \mathbb{E}\times (0,\bar{\gamma}] \to  \mathbb{P}(\mathbb{A})\times\mathbb{P}(\mathbb{B})$ and functions $v\in \mathcal{C}$ the operator
$$ (T_{\mu\nu} v) (x,\gamma):= (Lv)\big(x,\gamma,\mu(x,\gamma),\nu(x,\gamma)\big).$$
Also note that for $v\in\mathcal{C}$ it holds that there exists a (measurable) saddle point $(\mu^*,\nu^*)$ such that
$$ Lv\big(x,\gamma,\mu(x,\gamma),\nu_n^*(x,\gamma)\big) \le Lv\big(x,\gamma,\mu_n^*(x,\gamma),\nu_n^*(x,\gamma)\big)\le Lv\big(x,\gamma,\mu^*(x,\gamma),\nu(x,\gamma)\big)$$
or equivalently
$$ (T_{\mu\nu^*}v)(x,\gamma) \le (T_{\mu^*\nu^*}v)(x,\gamma) \le (T_{\mu^*\nu}v)(x,\gamma) $$
for all $(\mu,\nu) $ and that $(x,\gamma)\mapsto Lv\big(x,\gamma,\mu^*(x,\gamma),\nu^*(x,\gamma)\big)\in\mathcal{C}$.
This follows from a classical measurable selection theorem in \cite{bp73} and a minimax theorem (Theorem 2) in \cite{fan53}, see e.g. \cite{ri78}.

By induction on $n$ we show that
\begin{itemize}
  \item[(i)] $V_n = TV_{n-1}\in\mathcal{C},$
  \item[(ii)] $T_{\mu_n\nu_n^*}\ldots T_{\mu_1\nu_1^*} V_0 \le V_n$ for any measurable $\mu_1,\ldots ,\mu_n : \mathbb{E}\times (0,\bar{\gamma}] \to  \mathbb{P}(\mathbb{A})$,
  \item[(iii)] $T_{\mu_n^*\nu_n^*}\ldots T_{\mu_1^*\nu_1^*} V_0 = V_n,$
  \item[(iv)] $T_{\mu_n^*\nu_n}\ldots T_{\mu_1^*\nu_1} V_0 \ge  V_n$ for any measurable $\nu_1,\ldots ,\nu_n : \mathbb{E}\times (0,\bar{\gamma}] \to  \mathbb{P}(\mathbb{B})$,
 \end{itemize}
and finally that $T_{\mu_N^*\nu_N^*}\ldots T_{\mu_1^*\nu_1^*} V_0 = T_{f_0^*g_0^*}\ldots T_{f_{N-1}^*g_{N-1}^*} V_0 = V_{N\pi^*\sigma^*}.$

For $n=1$ we obtain by definition of $V_1$ and Theorem \ref{theo:iteration} that
$$ V_1 = \inf_g \sup_f T_{fg} V_0 = TV_0.$$
The remaining statements (ii)-(iv) follow directly from the definition of $\mu_1^*$, $\nu_1^*$ and $V_1$.
Now suppose the statement is true for $n-1$.  Obviously the $T_{\mu\nu}$-operator is monotone, i.e.\ for $v,w\in\mathcal{C}$ with
$v\le w$ we obtain $T_{\mu\nu} v \le T_{\mu\nu} w$. Since $V_{n-1}\in \mathcal{C}$, the selection theorem and the minimax theorem imply the existence of a saddle point $(\mu_n^*,\nu_n^*)$ on stage $n$.
With the induction hypothesis we obtain
\begin{eqnarray*}
  T_{\mu_n\nu_n^*}\ldots T_{\mu_1\nu_1^*} V_0 &\le& T_{\mu_n\nu_n^*} V_{n-1} \le  T_{\mu_n^*\nu_n^*} V_{n-1} =  T_{\mu_n^*\nu_n^*}\ldots T_{\mu_1^*\nu_1^*} V_0
\end{eqnarray*}
for any $\mu_1,\ldots ,\mu_n$.
On the other hand
\begin{eqnarray*}
  T_{\mu_n^*\nu_n^*} V_{n-1} &\le & T_{\mu_n^*\nu_n} V_{n-1}\le  T_{\mu_n^*\nu_n}\ldots T_{\mu_1^*\nu_1} V_0.
\end{eqnarray*}
for any  $\nu_1,\ldots ,\nu_n$. Moreover since we can identify a sequence $(\mu_1,\ldots ,\mu_n)$ with a policy ${\pi} =(f_0,\ldots,f_{n-1})$ by setting $f_k(x) := \mu_{n-k}(x,\gamma\beta^k)$ for $k=0,\ldots,n-1$ and similarly for ${\pi}^*,{\sigma},{\sigma}^*,$ we have shown that there exist policies ${\pi}^*$ and ${\sigma}^*$ for the $n$-stage game such that for any arbitrary policies ${\pi},{\sigma}$
$$ V_{n{\pi}{\sigma}^*}\le V_{n{\pi}^*{\sigma}^*} \le V_{n{\pi}^*{\sigma}}.$$
But this means that $({\pi}^*,{\sigma}^*)$ is a saddle-point and $V_n=V_{n{\pi}^*{\sigma}^*} = T_{\mu_n^*\nu_n^*}\ldots T_{\mu_1^*\nu_1^*} V_0$ is the value of the game and $\pi^*$ and $\sigma^*$ are optimal for player 1 and 2 respectively.

Combining above results, the statement follows.
\end{proof}

\begin{remark}\label{rem:beta1}
If $\beta=1$ we can skip the second component and simply write the Shapley equation as $V_n(x) =(TV_{n-1})(x) = \sup_\mu \inf_\nu \int V_{n-1}(x')\tilde{Q}(dx'|x,\mu,\nu)$.
\end{remark}

\begin{remark}\label{rem:CE}
With the algorithm in Theorem \ref{theo:finite} we finally compute the value $V_N(x,\gamma)$ in \eqref{eq:uppervalue}. However, the main interest is the certainty equivalent in \eqref{eq:CE}. Of course by setting $\tilde{V}_N(x,\gamma) := \frac1\gamma \ln V_N(x,\gamma)$ we obtain this value and we can easily modify the Shapley equation to obtain an equation for the certainty equivalent. Indeed it holds that
$$ \tilde{V}_n(x,\gamma) = \sup_\mu\inf_\nu \frac1\gamma \ln \int e^{\gamma \tilde{V}_{n-1}(x',\beta\gamma)} \tilde{Q}(dx'|x,\gamma,\mu,\nu).$$
\end{remark}

\begin{remark}
We have stated our results in the framework of $\gamma>0$, but the same arguments and statements also hold true in case $\gamma<0$. One only has to consider the state space $\mathbb{E}\times[\underline{\gamma},0)$.
\end{remark}

\begin{example}
In order to study the effect of the risk aversion parameter $\gamma$, we consider a simple one-period game where players have only two actions, i.e. $\mathbb{A}=\mathbb{B}=\{-1,+1\}$. In this situation we do not need a state and can skip $x$.
The risk-sensitive reward for player 1 under policies $\pi=(f_0):=(1-y,y)$ and $\sigma=(g_0):=(1-z,z)$ is then given by
\begin{eqnarray*} V_{1\pi\sigma}(\gamma) &:=& \Eop^{\pi\sigma}\Big[ \exp\Big( \gamma r(A_0,B_0)\Big) \Big]\\
& =& e^{\gamma r(-1,-1)}(1-y)(1-z)+ e^{\gamma r(-1,1)} (1-y)z+e^{\gamma r(1,-1)} y(1-z)+e^{\gamma r(1,1)}yz. \end{eqnarray*}
We assume that $\min\{r(-1,-1),r(1,1)\} > \max\{r(-1,1),r(1,-1)\}$. The general solutions for the optimal strategies for player 1 and 2 are:
\begin{eqnarray*}
    y^*&=& \frac{e^{\gamma r(-1,1)}-e^{\gamma r(-1,-1)}}{e^{\gamma r(-1,1)}+e^{\gamma r(1,-1)} -e^{\gamma r(-1,-1)}-e^{\gamma r(1,1)}},\\
    z^*&=& \frac{e^{\gamma r(1,-1)}-e^{\gamma r(-1,-1)}}{e^{\gamma r(-1,1)}+e^{\gamma r(1,-1)} -e^{\gamma r(-1,-1)}-e^{\gamma r(1,1)}}.
\end{eqnarray*}
These solutions depend on $\gamma\neq 0$. In our numerical example we have chosen $$r(-1,-1)=5, r(-1,1)=3, r(1,-1)=2, r(1,1)=4.$$ Figure \ref{fig:w} shows the optimal (randomized) strategies as functions of $\gamma$. The case $\gamma=0$ corresponds to the risk-neutral game with optimal strategies $y^*=0.5$ and $z^*=0.75$.
As we can see from Figure \ref{fig:w}, but can also be shown analytically, is that $y(\gamma) \to 1, z(\gamma)\to 1$ for $\gamma\to\infty$, i.e.\ in this case the players asymptotically use pure strategies and agree on the best value for player 1. For $\gamma\to-\infty$ we obtain $y(\gamma)\to 0, z(\gamma)\to 1$, i.e.\ again both players use asymptotically pure strategies and this time agree on the worst value for player 1. Intuitively this behavior is clear, since the randomization of the strategies is the only stochastic component in this model and whereas it is optimal to use randomized strategies in the risk-neutral case ($\gamma=0$) players tend to reduce their own produced variability if $|\gamma|\to\infty$. For the special model with countable state space in \cite{bg14}, the authors have shown in Section 5 that this observation is also true for the risk-sensitive average reward case. The behavior of the optimal strategies in the discounted reward case is still open.

\begin{figure}
      \hspace*{-0.5cm}  \includegraphics[trim = 10mm 30mm 10mm 30mm, height=0.50\textwidth]{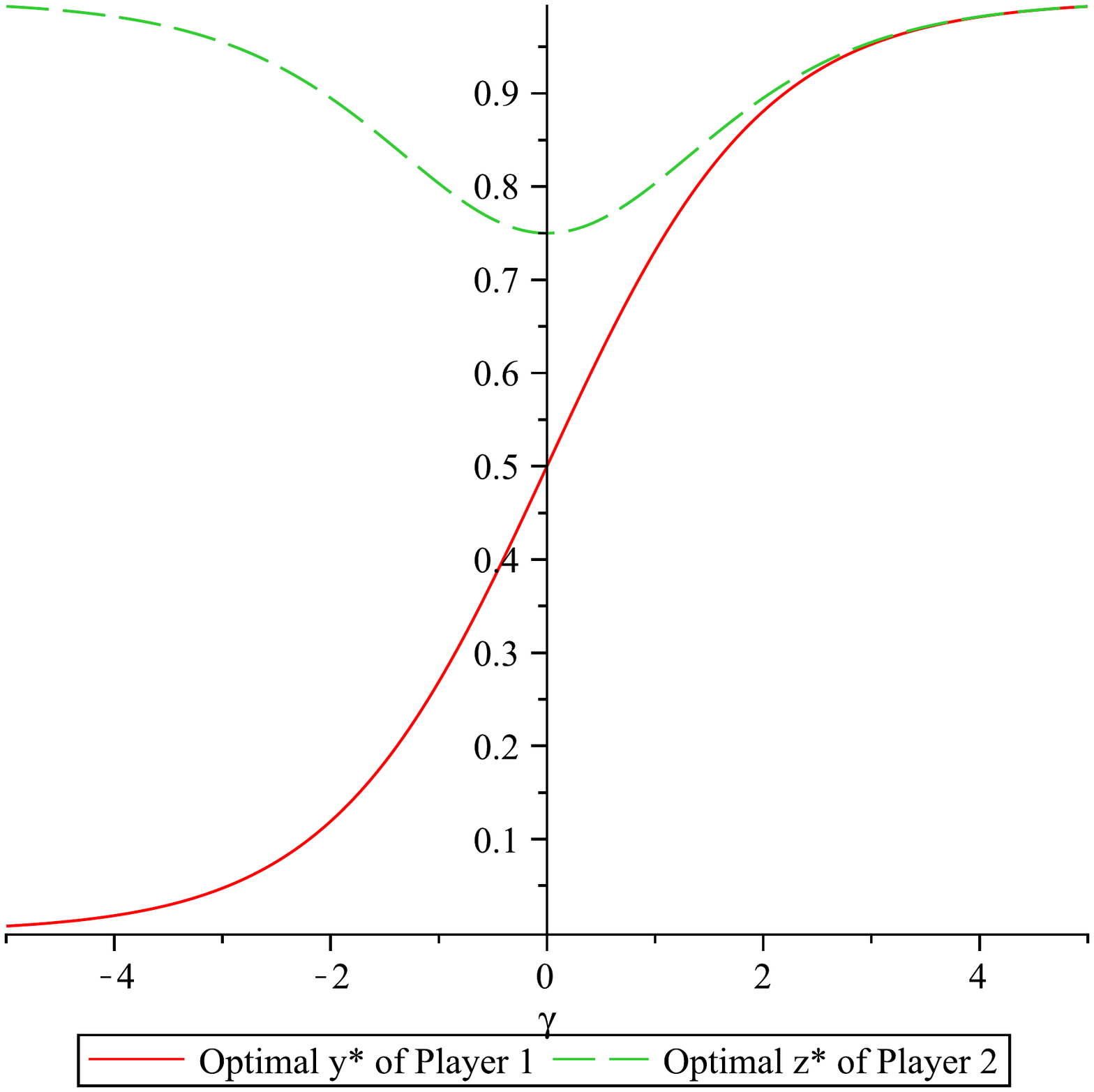}\hspace*{1cm}
      \includegraphics[trim = 10mm 30mm 10mm 30mm, height=0.5\textwidth]{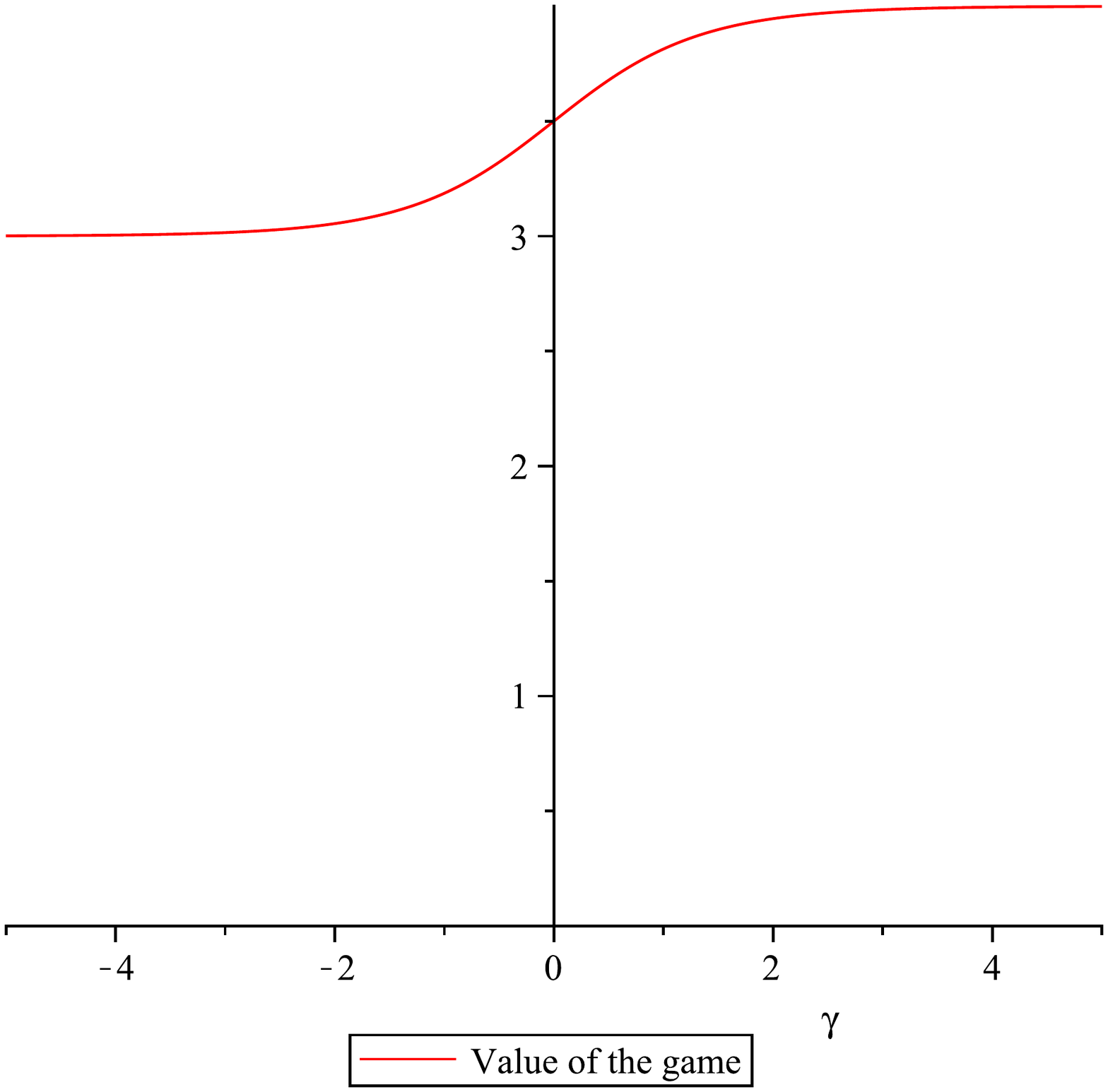}\hspace*{1cm}
    \caption{Optimal strategies for both players as a function of $\gamma\neq0$ (left) and value of the game as a function of $\gamma\neq0$ (right).}
    \label{fig:w}
\end{figure}

\end{example}

\section{Infinite Horizon Discounted Game}\label{sec:infh}
In this section we consider the game with an infinite time horizon and $\beta\in(0,1)$. The risk-sensitive reward for player 1 under the policies $\pi$ and $\sigma$ is here given by
\begin{equation} V_{\infty\pi\sigma}(x,\gamma) := \Eop_x^{\pi\sigma}\Big[ \exp\Big( \gamma \sum_{k=0}^{\infty} \beta^k r(X_k,A_k,B_k)\Big) \Big].\end{equation}
The corresponding {\em upper value} of the game is
\begin{equation} V_{\infty}(x,\gamma) := \inf_\sigma \sup_\pi V_{\infty\pi\sigma}(x,\gamma). \end{equation}
If $\inf_\sigma \sup_\pi V_{\infty\pi\sigma}(x,\gamma) = \sup_\pi  \inf_\sigma V_{\infty\pi\sigma}(x,\gamma)$, the function $V_\infty$ is called {\em value function} of the stochastic game with infinite time horizon.
A policy $\pi^*$ is called {\em optimal for player 1} in this case if
\begin{equation*} V_{\infty\pi^*\sigma}(x,\gamma) \ge \sup_\pi \inf_\sigma  V_{\infty\pi\sigma}(x,\gamma),\; \mbox{ for all } \sigma. \end{equation*}
 A policy $\sigma^*$ is called {\em optimal for player 2}  if
\begin{equation*} V_{\infty\pi\sigma^*}(x,\gamma) \le  \inf_\sigma \sup_\pi V_{\infty\pi\sigma}(x,\gamma),\; \mbox{ for all } \pi. \end{equation*}
The policy pair $(\pi^*,\sigma^*)$ for which
\begin{equation*}
    V_{\infty\pi\sigma^*} \le V_{\infty\pi^*\sigma^*} \le V_{\infty\pi^*\sigma},\quad\mbox{for all } \pi,\sigma
\end{equation*}
is called {\em saddle-point equilibrium} and implies that $\pi^*$ is optimal for player 1 and $\sigma^*$ is optimal for player 2.

\subsection{Positive risk-sensitivity $\gamma>0$}
Here we assume that $\gamma\in(0,\bar{\gamma}]$ and denote the upper bound of $r$ by $\bar{r}$. The next theorem provides the solution of the infinite horizon problem. Recall that $\mathcal{C}$ is the set of all continuous and bounded functions on $\mathbb{E}\times (0,\bar{\gamma}]$.

\begin{theorem}\label{theo:posra}
Assume (A1)-(A4) and let $\gamma\in(0,\bar{\gamma}]$.
\begin{itemize}
  \item[a)] It holds that $V_\infty$ is the unique solution $v\in\mathcal{C}$ of the Shapley equation $$V_\infty(x,\gamma) =  TV_\infty(x,\gamma)=  \inf_\nu \sup_\mu \int V_{\infty}(x',\beta \gamma) \tilde{Q}(dx'|x,\gamma,\mu,\nu)$$
  with $1 \le v(x,\gamma)\le e^{\gamma  \frac{\bar{r}}{1-\beta}}$.
  \item[b)]  There exist measurable functions  $(\mu^*,\nu^*):\mathbb{E}\times (0,\bar{\gamma}]\to \mathbb{P}(\mathbb{A})\times \mathbb{P}(\mathbb{B}) $ with $\mu^*(x,\gamma)\in \mathbb{P}(\mathbb{A}_x)$ and $\nu^*(x,\gamma)\in \mathbb{P}(\mathbb{B}_x)$ such that $(\mu^*(x,\gamma),\nu^*(x,\gamma))$ is a saddle point of
      $$ (\mu,\nu) \mapsto LV_{\infty}\big(x,\gamma,\mu,\nu\big)$$
   for all $ (\mu,\nu) \in  \mathbb{P}(\mathbb{A}_x)\times\mathbb{P}(\mathbb{B}_x)$. Then $V_\infty$ is the value of the infinite horizon stochastic game and $(\pi^*,\sigma^*)=({f}_n^*,{g}_n^*)_{n=0,\ldots,N-1}$ with ${f}_n^*(x) := \mu^*(x,\gamma\beta^n)$ and  ${g}_{n}^*(x) := \nu^*(x,\gamma\beta^n)$ are optimal policies for player 1 and 2 respectively.
\end{itemize}
\end{theorem}

\begin{proof}
\begin{itemize}
  \item[a)]
Obviously the values $V_n=T^nV_0$ of the $n$-stage games are increasing and bounded and hence a limit $V:= \lim_{n\to\infty} V_n$ exists. We first prove that $V_\infty=V$. By monotonicity we see directly that $V_n\le V_\infty$ which implies that $V\le V_\infty$. On the other hand we have for arbitrary policies $\pi$ and $\sigma$ that $V_{\infty \pi\sigma}(x,\gamma) \le V_{n\pi\sigma}(x,\gamma) \cdot \delta_n$ with $\delta_n := \exp\Big(\gamma\beta^n \frac{\bar{r}}{1-\beta}\Big)$ and $\lim_{n\to\infty}\delta_n=1$. Thus we obtain
$$ V_\infty(x,\gamma)=\inf_\sigma\sup_\pi V_{\infty \pi\sigma}(x,\gamma) \le \inf_\sigma\sup_\pi V_{n\pi\sigma}(x,\gamma) \cdot \delta_n = V_n(x,\gamma) \cdot \delta_n$$
which yields $V_\infty \le V$ and thus equality. Note that the convergence $\lim_{n\to\infty}V_n=V_\infty$ is uniform because of
$$ 0\le V_\infty(x,\gamma)-V_n(x,\gamma) \le V_n(x,\gamma) \cdot(\delta_n-1 )\le e^{\bar{\gamma} \frac{\bar{r}}{1-\beta}} \cdot\Big( \exp\Big(\bar{\gamma}\beta^n \frac{\bar{r}}{1-\beta}\Big) -1 \Big).$$
The right hand side converges to $0$ for $n\to\infty$ and does not depend on $x$ and $\gamma$ any more. Uniform convergence implies that $V_\infty\in\mathcal{C}$ again.

Next we show that $V_\infty = TV_\infty$. Since $V_n\le V_\infty$ we obtain by applying the $T$-operator to this inequality that $V_{n+1}=TV_n \le TV_\infty$ and by letting $n\to\infty$ that $V_\infty=V \le  TV_\infty$. On the other hand we have $V_\infty \le V_n\cdot \delta_n $ and again by applying the $T$-operator to this inequality we obtain $TV_\infty \le T(V_n\cdot \delta_n) = \delta_n\cdot TV_n = \delta_n \cdot V_{n+1}$. Thus by letting $n\to\infty$ we obtain $TV_\infty \le V_\infty$ and hence $V_\infty=TV_\infty$.

Finally we show the uniqueness of the solution of the Shapley equation in the set of continuous functions with lower bound $\underline{b}:= 1$ and upper bound $\bar{b}(\gamma) := e^{\gamma  \frac{\bar{r}}{1-\beta}}$. The first step is to see that $T\underline{b} \ge \underline{b}$ because $\tilde{Q}$ is not necessarily a probability measure but has a total mass larger or equal to $1$. On the other hand we have
$$ T\bar{b}(x,\gamma) = \inf_\nu\sup_\mu \int \int \bar{b}(\beta\gamma) e^{\gamma r(x,a,b)} \mu(da) \nu(db) \le e^{\gamma\bar{r}} e^{\gamma\beta\frac{\bar{r}}{1-\beta}}=\bar{b}(\gamma).$$
Since the $T$-operator is monotone we have that $T^n\underline{b} \uparrow$ and $T^n\bar{b}\downarrow$ for $n\to\infty$. Obviously we obtain
\begin{eqnarray*}
  T^n\underline{b}(x,\gamma) &=& V_n(x,\gamma) = \inf_\sigma\sup_\pi  \Eop_x^{\pi\sigma}\Big[ \exp\Big( \gamma \sum_{k=0}^{n-1} \beta^k r(X_k,A_k,B_k)\Big) \Big] \\
  T^n\bar{b}(x,\gamma) &=& \delta_n V_n(x,\gamma)=\delta_n  T^n\underline{b}(x,\gamma).
\end{eqnarray*}
Hence we see that  $T^n\underline{b} \uparrow V_\infty$ and $T^n\bar{b}\downarrow V_\infty$ for $n\to\infty$. Now suppose that there is another solution $v\in\mathcal{C}$ of $v=Tv$ with $\underline{b}\le v\le \bar{b}$. This then implies that $T^n\underline{b}  \le v\le T^n\bar{b}$ for all $n$ which shows uniqueness of the fixed point.
\item[b)] The existence of a saddle point $(\mu^*,\nu^*)$ follows again from the measurable selection theorem and the minimax theorem. Let $T_{\mu\nu} v$ be defined as in Theorem \ref{theo:finite}. By monotonicity and the fact that $1\le V_\infty(x,\gamma) \le e^{\gamma  \frac{\bar{r}}{1-\beta}}$ we obtain that $\lim_{n\to\infty}T^n_{\mu^*\nu^*} V_0 = \lim_{n\to\infty}T^n_{\mu^*\nu^*} V_\infty= V_{\infty\pi^*\sigma^*}$ for $\pi^*,\sigma^*$ defined in the statement. By the definition of the saddle point we obtain for any admissible $(\mu_1,\nu_1) : \mathbb{E}\times (0,\bar{\gamma}] \to  \mathbb{P}(\mathbb{A})\times \mathbb{P}(\mathbb{B}) $ that
    $$ T_{\mu_1\nu^*} V_\infty\le T_{\mu^*\nu^*} V_\infty \le T_{\mu^*\nu_1} V_\infty.$$
As usual the saddle point property implies that $\sup_\mu\inf_\nu T_{\mu\nu} V_\infty = \inf_\nu\sup_\mu T_{\mu\nu} V_\infty = TV_\infty = V_\infty$ and $T_{\mu^*\nu^*} V_\infty=V_\infty$. Hence we can also write
 $$ T_{\mu_1\nu^*} V_\infty\le V_\infty \le T_{\mu^*\nu_1} V_\infty.$$
 By iterating this inequality $n$-times we end up with
  $$ T_{\mu_1\nu^*}\ldots T_{\mu_n\nu^*}  V_\infty\le  V_\infty \le T_{\mu^*\nu_1}\ldots T_{\mu^*\nu_n} V_\infty$$
  for arbitrary $\mu_1,\ldots \mu_n$ and $\nu_1,\ldots,\nu_n$. Letting $n\to\infty$ implies since $1\le V_\infty(x,\gamma) \le e^{\gamma  \frac{\bar{r}}{1-\beta}}$ that $\lim_{n\to\infty} T_{\mu^*\nu_1}\ldots T_{\mu^*\nu_n} V_\infty = \lim_{n\to\infty} T_{\mu^*\nu_1}\ldots T_{\mu^*\nu_n} V_0=V_{\infty\pi^*\sigma}$ for $\sigma=(g_1,g_2,\ldots)$ with $g_n(x)=\nu_n(x,\gamma\beta^n)$. The same is true on the left-hand side. In total we obtain that$$ V_{\infty\pi\sigma^*} \le V_{\infty\pi^*\sigma^*} = V_\infty\le V_{\infty\pi^*\sigma}$$ for all policies $\pi$ and $\sigma$ which yields the statement.
\end{itemize}

\end{proof}

\subsection{Negative risk-sensitivity $\gamma<0$}
Here we assume that $\gamma\in[\bar{\gamma},0)$. The next theorem provides the solution of the infinite horizon problem.

\begin{theorem}
Let $\gamma\in[\bar{\gamma},0)$. Theorem \ref{theo:posra} holds as before with bounds $e^{\gamma  \frac{\bar{r}}{1-\beta}} \le v(x,\gamma)\le 1$ in part a).
\end{theorem}

\begin{proof}
Note that most inequalities simply reverse. In particular we obtain $V_\infty(x,\gamma) \ge V_n(x,\gamma)\cdot\delta_n$ with $\delta_n$ as in the previous proof.

\end{proof}

\section{Ergodic Game}\label{sec:arg}
In this section we consider the risk-sensitive average reward with $\beta=1$. Following Remark \ref{rem:beta1} we now skip the dependence on $\gamma$ in our notations. Hence for two randomized policies $\pi$ and $\sigma$ for player 1 and 2 respectively we define
\begin{eqnarray}\label{eq:AC}
    J_{\pi\sigma}(x) &:= &\limsup_{n\to\infty} \frac1n \frac1\gamma \ln \Eop_x^{\pi\sigma}\Big[\exp(\gamma \sum_{k=0}^{n-1} r(X_k,A_k,B_k))\Big] \\
    J(x) &:=& \inf_\sigma\sup_\pi J_{\pi\sigma}(x).
\end{eqnarray}
In case we have stationary policies, i.e.\ $\pi=(f,f,\ldots)$ and $\sigma=(g,g,\ldots)$ we simply write $J_{fg}(x)$.

If $\inf_\sigma \sup_\pi J_{\pi\sigma}(x) = \sup_\pi  \inf_\sigma J_{\pi\sigma}(x)$, the function $J$ is called {\em value function} of the average reward stochastic game.
A policy $\pi^*$ is called {\em optimal for player 1} in this case if
\begin{equation*} J_{\pi^*\sigma}(x) \ge \sup_\pi \inf_\sigma  J_{\pi\sigma}(x),\; \mbox{ for all } \sigma. \end{equation*}
 A policy $\sigma^*$ is called {\em optimal for player 2}  if
\begin{equation*} J_{\pi\sigma^*}(x) \le  \inf_\sigma \sup_\pi J_{\pi\sigma}(x),\; \mbox{ for all } \pi. \end{equation*}
The policy pair $(\pi^*,\sigma^*)$ for which
\begin{equation*}
    J_{\pi\sigma^*} \le J_{\pi^*\sigma^*} \le J_{\pi^*\sigma},\quad\mbox{for all } \pi,\sigma
\end{equation*}
is called {\em saddle-point equilibrium} and implies that $\pi^*$ is average optimal for player 1 and $\sigma^*$ is average optimal for player 2.

Here we cannot ignore the logarithm (see Remark \eqref{rem:CE}) and define for $f\in\mathbb{F}, g\in \mathbb{G}$, $\mu\in\mathbb{P}(\mathbb{A}_x), \nu\in \mathbb{P}(\mathbb{B}_x)$ and a measurable function $v:\mathbb{E}\to\R$ the following operators, given the expectation exists:
\begin{eqnarray*}
  (U_{fg}v)(x) &:=& \frac1\gamma \ln \int e^{\gamma v(x')} \tilde{Q}(dx'|x,f(x),g(x)),\\
  (Uv)(x) &:=& \inf_g\sup_f (U_{fg}v)(x)
\end{eqnarray*}
Recall the definition of $\tilde{Q}$ in \eqref{eq:tildeQ3}. Since $\gamma$ is fixed here we suppress the dependence on $\gamma$. Obviously $\tilde{Q}$ is in general not a probability measure. The normalizing constant is for $ x\in\mathbb{E}, \mu\in\mathbb{P}(\mathbb{A}_x)$ and $\nu\in \mathbb{P}(\mathbb{B}_x)$ given by
$$ c(x,\mu,\nu) := \int_\mathbb{E} \tilde{Q}(dx'|x,\mu,\nu) = \int\int e^{\gamma r(x,a,b)} \mu(da)\nu(db).$$
Since $0\le r\le \bar{r}$, the function $c$ is also bounded. More precisely $1\le c(x,\mu,\nu)\le e^{\gamma \bar{r}}$ for all $x,\mu,\nu$ in case $\gamma>0$ and $e^{\gamma \bar{r}}\le c(x,\mu,\nu) \le 1$ in case $\gamma<0$. Thus
$$ \hat{Q}(\cdot|x,\mu,\nu) := \frac{\tilde{Q}(\cdot|x,\mu,\nu)}{c(x,\mu,\nu)}$$
defines a transition kernel and using the notation $\hat{r}(x,\mu,\nu) := \frac1\gamma \ln c(x,\mu,\nu)$ we can rewrite the $U$-operator as
\begin{equation}\label{eq:Uop2}
   (Uv)(x) = \inf_\nu\sup_\mu \Big\{\hat{r}(x,\mu,\nu) + \frac1\gamma \ln \int e^{\gamma v(x')} \hat{Q}(dx'|x,\mu,\nu) \Big\}.
\end{equation}
Note that $0\le \hat{r}\le \bar{r}$.
Due to the dual representation of the exponential certainty equivalent (see e.g. Lemma 3.3 in \cite{hhm96}) it is possible to write
\begin{equation}\label{eq:Uop3}
   (Uv)(x) = \inf_\nu\sup_\mu\sup_\psi \Big\{\hat{r}(x,\mu,\nu) + \int v(x') \psi(dx') -\frac1\gamma I\big(\psi, \hat{Q}(\cdot|x,\mu,\nu)\big) \Big\}.
\end{equation}
where the supremum is over all probability measures $\psi\in\mathbb{P}(\mathbb{E})$ and $I(p,q)$ is the relative entropy of the two probability measures $p,q$ which is defined by
\begin{equation*}
    I(p,q) := \int \ln \frac{dp}{dq} p(dx)
\end{equation*}
when $p\ll q$ and $+\infty$ otherwise.

Note that the maximal probability measure in \eqref{eq:Uop3} is given by
\begin{equation}\label{eq:relentr}
    \psi_v(B|x,\mu,\nu) :=  \frac{\int_B e^{\gamma v(x')}\hat{Q}(dx'|x,\mu,\nu)}{\int_\mathbb{E} e^{\gamma v(x')}\hat{Q}(dx'|x,\mu,\nu)}
\end{equation}
for measurable sets $B\subset \mathbb{E}$, provided the denominator is finite. Obviously $\psi_v$ can be interpreted as a transition kernel.

In what follows let $w:\mathbb{E}\to[1,\infty)$ be a measurable weight function and define for measurable functions $v:\mathbb{E}\to\R$ the weighted supremum norm by
$$ \|v\|_{w} := \sup_{x\in \mathbb{E}}\frac{|v(x)|}{w(x)}.$$
By $\mathcal{B}_{w}$ we denote the space of all measurable functions $v:\mathbb{E}\to\R$ with finite $w$-norm.
We shall also consider the weighted span (semi) norm
$$ \|v\|_{sp,w} := \sup_{x,y\in \mathbb{E}}\frac{v(x)-v(y)}{w(x)+w(y)}.$$
The norms are related as follows (for a proof see \cite{HM}, Lemma 2.1):

\begin{lemma}\label{lem:norm}
For all $v\in \mathcal{B}_{w}$ we have $\|v\|_{sp,w} = \inf_{c\in\R} \|v+c\|_{w}$.
\end{lemma}

Let $W(x)\ge 0$ be measurable and $w(x) := 1+\beta W(x)$ for some $\beta>0$ (the discount factor is here equal to $1$, so we can use $\beta$ for a new variable). Obviously the norms $\|\cdot\|_{1+\beta W}$ are equivalent for all $\beta>0$. Hence $\mathcal{B}_{1+\beta W}$ is independent of $\beta$. Also the span (semi) norms $\|\cdot\|_{sp,1+\beta W}$ are equivalent for all $\beta>0$.

For our main results we use the following {\em ergodicity conditions} (E):

\begin{description}{\em
  \item[(E1)]  There exists a measurable function $W:\mathbb{E}\to [0,\infty)$ and constants $K >0$, $\bar{\alpha}\in(0,1)$ such that
  \begin{equation*}
  \int W(x') \psi_v(dx'|x,\mu,\nu) \le \bar{\alpha} W(x) +K\end{equation*}
  for all $v\in \mathcal{B}_{w}$ and $\mu \in \mathbb{P}(\mathbb{A}_x), \nu \in \mathbb{P}(\mathbb{B}_x)$ where  $\psi_v$ is given by \eqref{eq:relentr}.
    \item[(E2)] There exists a probability measure $\lambda$ and a constant $\tilde{\alpha}\in(0,1)$ such that
  $$   \inf_{x\in \mathbb{E}_0, \mu \in \mathbb{P}(\mathbb{A}_x), \nu \in \mathbb{P}(\mathbb{B}_x)} Q(\cdot|x,\mu,\nu) \ge \tilde{\alpha}\lambda(\cdot)$$
  where $\mathbb{E}_0:= \{x\in \mathbb{E} : W(x) \le R\}$ for some $R>\frac{2K}{1-\bar{\alpha}}$.}
\end{description}

\begin{remark}
\begin{itemize}
  \item[a)] Note that for $\gamma=0$ the ergodicity conditions (E) coincide with the ergodicity conditions in \cite{HM}.
  \item[b)] Assumption (E2) is a local minorization property which implies ergodicity of the state process together with the geometric ergodicity condition (E1). When $W\equiv 0$, then $\mathbb{E}_0=\R$ and (E2) becomes a global Doeblin condition.
  \item[c)] Note that (E2) is equivalent to
   $$  Q(\cdot|x,a,b) \ge \tilde{\alpha}\lambda(\cdot)$$
   for all  $x\in\mathbb{E}_0$ and for all $a \in \mathbb{A}_x, b \in \mathbb{B}_x$.
\end{itemize}
\end{remark}

Let us now define $w_0(x) := 1 +\frac1K W(x), x\in \mathbb{E}$. In order that $\psi_v$ is well-defined we make the following {\em integrability assumption} (F):
\begin{description}
  \item[(F)] {\em There exists a constant $K_0>0$ such that
  \begin{eqnarray*}
      &&\int e^{K_0w_0(x')} Q(dx'|x,\mu,\nu) <\infty,
\end{eqnarray*}
for all $x\in \mathbb{E}, \mu \in \mathbb{P}(\mathbb{A}_x), \nu \in \mathbb{P}(\mathbb{B}_x)$ and
  \begin{eqnarray}\label{eq.finite2}
      && \sup_{x\in \mathbb{E}_0, \mu \in \mathbb{P}(\mathbb{A}_x), \nu \in \mathbb{P}(\mathbb{B}_x)}\int e^{K_0w_0(x')} Q(dx'|x,\mu,\nu) <\infty.
\end{eqnarray}}
\end{description}


For $\bar{\gamma}>0$ (arbitrarily large) and $M>0$ define
$$ \gamma_0 := \min\{\bar{\gamma},\frac{K_0}M\}$$
and $\mathcal{B}_{w}^{(M)} := \{v\in \mathcal{B}_w : \|v\|_{sp,w}\le M\}$. Then $\psi_v$ is well-defined for all $v\in\mathcal{B}_{w}^{(M)}, \beta \in (0,\frac{1}{K})$ and $\gamma\in (-\gamma_0,\gamma_0)$, since
\begin{eqnarray*}
   && \int_\mathbb{E} e^{\gamma v(x')} \hat{Q}(dx'|x,\mu,\nu) \le  const.\cdot  \int_\mathbb{E} e^{K_0 w_0(x')} Q(dx'|x,\mu,\nu)<\infty.
\end{eqnarray*}

We show next that condition (E2) also holds for the transition kernels ${\psi}_v$:

\begin{lemma}\label{lem:unten}
Let  $\gamma\in (-\gamma_0,\gamma_0)$, $\beta\in(0,\frac1K)$ and assume (F). Then condition (E2) implies that there exists a probability measure $\tilde{\lambda}$ and a constant $\underline{\alpha}\in(0,1)$ such that
  $$  \inf_{x\in \mathbb{E}_0, \mu \in \mathbb{P}(\mathbb{A}_x), \nu \in \mathbb{P}(\mathbb{B}_x)} \psi_v(\cdot|x,\mu,\nu) \ge \underline{\alpha}\tilde{\lambda}(\cdot)$$
 for all $v\in \mathcal{B}_w^{(M)}$.
\end{lemma}

\begin{proof}
First note that due to (E2) we have
\begin{eqnarray*}
   && \inf_{x\in \mathbb{E}_0, a\in \mathbb{A}_x, b\in \mathbb{B}_x} \int_B e^{ -K_0w_0(x')} Q(dx'|x,a,b) \ge \tilde{\alpha}\int_B e^{-K_0 w_0(x')} \lambda(dx') =: \tilde{\alpha}\int_\mathbb{E} e^{-K_0w_0(x')} \lambda(dx') \cdot \tilde{\lambda}(B).
\end{eqnarray*}
 We consider now the case $\gamma\in(0,\gamma_0)$. The case $\gamma\in (-\gamma_0,0)$ is similar. By definition of $\hat{Q}$ we obtain for $v\in \mathcal{B}_w^{(M)}, \mu \in \mathbb{P}(A_x), \nu \in \mathbb{P}(B_x)$:
\begin{eqnarray*}
  {\psi}_v(B|x,\mu,\nu) &\ge& \frac{1}{e^{\gamma \bar{r}}} \frac{\int_B e^{\gamma v(x')} Q(dx'|x,\mu,\nu)}{\int_\mathbb{E} e^{\gamma v(x')} Q(dx'|x,\mu,\nu)} \\
   &\ge &  \frac{1}{e^{\bar{\gamma}\bar{r}}}  \frac{\int_B e^{-K_0w_0(x')} Q(dx'|x,\mu,\nu)}{\int_\mathbb{E} e^{K_0w_0(x')} Q(dx'|x,\mu,\nu)}\\
   &\ge & \underline{\alpha} \tilde{\lambda}(B),
\end{eqnarray*}
where $\underline{\alpha}$ is defined by
$$ \underline{\alpha}  := \frac{\tilde{\alpha}}{e^{\bar{\gamma}\bar{r}}}\frac{\int_\mathbb{E} e^{-K_0w_0(x')} \lambda(dx')}{  \sup_{x\in \mathbb{E}_0, \mu \in \mathbb{P}(\mathbb{A}_x), \nu \in \mathbb{P}(\mathbb{B}_x)}\int e^{K_0w_0(x')} Q(dx'|x,\mu,\nu) }<1$$
and is thus independent of $M$ and $\beta$.  
\end{proof}


Moreover, we need the following constants (see \cite{HM}): 
For $ \delta_0 := \bar{\alpha}+2\frac{K}R <1$ and $\alpha_0\in(0,\underline{\alpha})$ let $$\beta := \frac{\alpha_0}{K},\quad \alpha:= (1-\underline{\alpha}+\alpha_0)\vee \frac{2+R{\beta}\delta_0}{2+R{\beta}}\quad \mbox{ and }\quad M:= \frac{\bar{r}}{1-\alpha}.$$
From now on we consider $w(x) = 1+\beta W(x)$ with $W$ from (E1) and $\beta :=\frac{\alpha_0}{K}<\frac1K$.
The next theorem is crucial for the solution of the average risk-sensitive game.

\begin{theorem}\label{theo:Ucontract}
Let  $\gamma\in (-\gamma_0,\gamma_0)$ and assume (E) and (F). Then  $U: \mathcal{B}_w \to \mathcal{B}_w$ and for all $v_1, v_2\in\mathcal{B}_w^{(M)}$
 $$ \| Uv_1-Uv_2\|_{sp,w} \le \alpha \|v_1-v_2\|_{sp,w}$$
where $\alpha\in(0,1)$ has been defined above.
\end{theorem}

\begin{proof}
The fact that  $U: \mathcal{B}_w \to \mathcal{B}_w$ follows from  (E1).
For the second statement we use the representation of $U$ in \eqref{eq:Uop3}. Let $v_1,v_2\in\mathcal{B}_w^{(M)}$ and $x_1,x_2\in\mathbb{E}$. In order to obtain the right estimate we define the following $\varepsilon$-minimizer or maximizer respectively for an $\varepsilon>0$:
\begin{eqnarray*}
  \nu_1 &:=& \varepsilon-argmin_{\nu} \Big\{ \sup_\mu \sup_\psi \big\{ \hat{r}(x_1,\mu,\nu) + \int v_2(x')\psi(dx'|x_1)-\frac1\gamma I\big(\psi,\hat{Q}(\cdot|x_1,\mu,\nu)\big\}\Big\}\\
  \mu_1 &:=& \varepsilon-argmax_{\mu} \Big\{ \sup_\psi \big\{ \hat{r}(x_1,\mu,\nu_1) + \int v_1(x')\psi(dx'|x_1)-\frac1\gamma I\big(\psi,\hat{Q}(\cdot|x_1,\mu,\nu_1)\big\}\Big\}\\
  \nu_2 &=& \varepsilon-argmin_{\nu} \Big\{ \sup_\mu \sup_\psi \big\{ \hat{r}(x_2,\mu,\nu) + \int v_1(x')\psi(dx'|x_2)-\frac1\gamma I\big(\psi,\hat{Q}(\cdot|x_2,\mu,\nu)\big\}\Big\} \\
  \mu_2 &=&  \varepsilon-argmax_{\mu} \Big\{ \sup_\psi \big\{ \hat{r}(x_2,\mu,\nu_2) + \int v_2(x')\psi(dx'|x_2)-\frac1\gamma I\big(\psi,\hat{Q}(\cdot|x_2,\mu,\nu_2)\big\}\Big\}.
\end{eqnarray*}
Then we obtain that
\begin{eqnarray*}
 && (Uv_1)(x_1)-(Uv_2)(x_1)-\big( (Uv_1)(x_2)-(Uv_2)(x_2)\big) \\
 &\le&  \sup_\mu\sup_\psi \Big\{\hat{r}(x_1,\mu,\nu_1) + \int v_1(x') \psi(dx'|x_1) -\frac1\gamma I\big(\psi, \hat{Q}(\cdot|x_1,\mu,\nu_1)\big) \Big\}\\
 && - \sup_\mu\sup_\psi \Big\{\hat{r}(x_1,\mu,\nu_1) + \int v_2(x') \psi(dx'|x_1) -\frac1\gamma I\big(\psi, \hat{Q}(\cdot|x_1,\mu,\nu_1)\big) \Big\}\\
 && - \sup_\mu\sup_\psi \Big\{\hat{r}(x_2,\mu,\nu_2) + \int v_1(x') \psi(dx'|x_2) -\frac1\gamma I\big(\psi, \hat{Q}(\cdot|x_2,\mu,\nu_2)\big) \Big\}\\
 && + \sup_\mu\sup_\psi \Big\{\hat{r}(x_2,\mu,\nu_2) + \int v_2(x') \psi(dx'|x_2) -\frac1\gamma I\big(\psi, \hat{Q}(\cdot|x_2,\mu,\nu_2)\big) \Big\}-2\varepsilon\\
 &\le &  \sup_\psi \Big\{\hat{r}(x_1,\mu_1,\nu_1) + \int v_1(x') \psi(dx'|x_1) -\frac1\gamma I\big(\psi, \hat{Q}(\cdot|x_1,\mu_1,\nu_1)\big) \Big\}\\
 && - \sup_\psi \Big\{\hat{r}(x_1,\mu_1,\nu_1) + \int v_2(x') \psi(dx'|x_1) -\frac1\gamma I\big(\psi, \hat{Q}(\cdot|x_1,\mu_1,\nu_1)\big) \Big\}\\
 && - \sup_\psi \Big\{\hat{r}(x_2,\mu_2,\nu_2) + \int v_1(x') \psi(dx'|x_2) -\frac1\gamma I\big(\psi, \hat{Q}(\cdot|x_2,\mu_2,\nu_2)\big) \Big\}\\
 && + \sup_\psi \Big\{\hat{r}(x_2,\mu_2,\nu_2) + \int v_2(x') \psi(dx'|x_2) -\frac1\gamma I\big(\psi, \hat{Q}(\cdot|x_2,\mu_2,\nu_2)\big) \Big\}-4\varepsilon\\
& \le&  \int \big( v_1(x') - v_2(x')\big)  \psi_{v_1}(dx'|x_1,\mu_1,\nu_1)-   \int \big( v_1(x') - v_2(x')\big) \psi_{v_2}(dx'|x_2,\mu_2,\nu_2)-4\varepsilon.
\end{eqnarray*}
Then we let $\varepsilon\downarrow 0$ and proceed with the inequality as follows, where $c\in\R$ is arbitrary:
\begin{eqnarray*}
   && \int \big( v_1(x') - v_2(x')\big)  \psi_{v_1}(dx'|x_1,\mu_1,\nu_1)-  \int \big( v_1(x') - v_2(x')\big) \psi_{v_2}(dx'|x_2,\mu_2,\nu_2) =\\
   &=& \int \big( v_1(x') - v_2(x')+c\big)  \psi_{v_1}(dx'|x_1,\mu_1,\nu_1)-  \int \big( v_1(x') - v_2(x')+c\big) \psi_{v_2}(dx'|x_2,\mu_2,\nu_2) \\
   &=& \int \frac{ v_1(x') - v_2(x')+c}{w(x')} w(x')  \psi_{v_1}(dx'|x_1,\mu_1,\nu_1)-
    \int \frac{ v_1(x') - v_2(x')+c}{w(x')} w(x')  \psi_{v_2}(dx'|x_2,\mu_2,\nu_2).
\end{eqnarray*}
Now we go ahead as in \cite{HM} and distinguish two cases:

Case 1: $W(x_1)+W(x_2) \ge R$.

Here we obtain
\begin{eqnarray*}
   && \int \frac{ v_1(x') - v_2(x')+c}{w(x')} w(x') \psi_{v_1}(dx'|x_1,\mu_1,\nu_1)- \int \frac{ v_1(x') - v_2(x')+c}{w(x')} w(x')  \psi_{v_2}(dx'|x_2,\mu_2,\nu_2)  \\
   &\le& \| v_1-v_2+c\|_{w} \cdot \Big( \int w(x') \psi_{v_1}(dx'|x_1,\mu_1,\nu_1)+ \int w(x') \psi_{v_2}(dx'|x_2,\mu_2,\nu_2)\Big)
  \end{eqnarray*}
Taking the infimum over all $c\in\R$ we obtain the first inequality below using Lemma \ref{lem:norm}. The remainder follows from (E1) as in the proof of Theorem 3.1 in \cite{HM}:
\begin{eqnarray*}
   &&  \int \big( v_1(x') - v_2(x')\big) \psi_{v_1}(dx'|x_1,\mu_1,\nu_1)-  \int \big( v_1(x') - v_2(x')\big) \psi_{v_2}(dx'|x_2,\mu_2,\nu_2) \\
   &\le& \| v_1-v_2\|_{sp,w} \cdot \Big(\int w(x') d\psi_{v_1}(dx'|x_1,\mu_1,\nu_1)+ \int w(x') \psi_{v_2}(dx'|x_2,\mu_2,\nu_2)\Big)\\
   &\le & \| v_1-v_2\|_{sp,w} \cdot \Big(2+\beta\bar{\alpha} W(x_1)+\beta\bar{\alpha} W(x_2)+2\beta K\Big)\\
   &\le& \| v_1-v_2\|_{sp,w} \cdot \Big(2+\beta\delta_0 W(x_1)+\beta\delta_0 W(x_2)\Big)\\
     & \le & \| v_1-v_2\|_{sp,w} \cdot \delta_1 \Big(2+\beta W(x_1)+\beta W(x_2)\Big)
  \end{eqnarray*}
  with $\delta_1 := \frac{2+R{\beta}\delta_0}{2+R{\beta}} <1$.
Altogether the claim follows.

Case 2: $W(x_1)+W(x_2) \le R$. Then $x_1,x_2\in \mathbb{E}_0$.

Here we define for $i=1,2$ using Lemma \ref{lem:unten}:
$$ \tilde{\psi}_{v_i}(\cdot |x_i,\mu_i,\nu_i) = \frac{1}{1-\underline{\alpha}} \psi_{v_i}(\cdot |x_i,\mu_i,\nu_i)-\frac{\underline{\alpha}}{1-\underline{\alpha}} \tilde{\lambda}(\cdot).$$
Hence for a measurable function $h:\mathbb{E}\to\R$ we obtain
$$ \int h(x')  \psi_{v_i}(dx'|x_i,\mu_i,\nu_i) = (1-\underline{\alpha}) \int h(x')  \tilde{\psi}_{v_i}(dx'|x_i,\mu_i,\nu_i)+\underline{\alpha} \int h(x') d\tilde{\lambda}.$$
Then we get in the same way as in \cite{HM} Theorem 3.1:
\begin{eqnarray*}
   && \int \frac{ v_1(x') - v_2(x')}{w(x')} w(x') \psi_{v_1}(dx'|x_1,\mu_1,\nu_1)-\int \frac{ v_1(x') - v_2(x')}{w(x')} w(x')  \psi_{v_2}(dx'|x_2,\mu_2,\nu_2)  \\
   &=& (1-\underline{\alpha}) \int \frac{ v_1(x') - v_2(x')}{w(x')} w(x')  \tilde{\psi}_{v_1}(dx'|x_1,\mu_1,\nu_1)-\\
   && -(1-\underline{\alpha})   \int \frac{ v_1(x') - v_2(x')}{w(x')} w(x')   \tilde{\psi}_{v_2}(dx'|x_2,\mu_2,\nu_2)\\
   &\le &  \| v_1-v_2\|_{sp,w} (1-\underline{\alpha}) \cdot \Big(\int w(x') d \tilde{\psi}_{v_1}(dx'|x_1,\mu_1,\nu_1)+ \int w(x')  \tilde{\psi}_{v_2}(dx'|x_2,\mu_2,\nu_2)\Big)\\
    &\le &  \| v_1-v_2\|_{sp,w} \cdot \Big(2 (1-\underline{\alpha})+ \beta\bar{\alpha} W(x_1)+\beta\bar{\alpha} W(x_2)+2\beta K\Big).
  \end{eqnarray*}
Recalling the definition of $\beta$ 
and setting $\delta_2:= (1-(\underline{\alpha}-\alpha_0)) \vee \bar{\alpha}$ yields
\begin{eqnarray*}
   &&  \int \big( v_1(x') - v_2(x')\big) \psi_{v_1}(dx'|x_1,\mu_1,\nu_1)-  \int \big( v_1(x') - v_2(x')\big) \psi_{v_2}(dx'|x_2,\mu_2,\nu_2)  \\
   &\le& \| v_1-v_2\|_{sp,w} \cdot \Big(2(1-(\underline{\alpha}-\alpha_0)) + \beta\bar{\alpha} W(x_1)+\beta\bar{\alpha} W(x_2)\Big)\\
   &\le & \| v_1-v_2\|_{sp,w} \cdot \delta_2 \Big(2+\beta W(x_1)+\beta W(x_2)\Big).
  \end{eqnarray*}
Combining above results, the statement follows.

\end{proof}

Next we show the following lemma.

\begin{lemma}\label{lem:v12}
Assume (E) and (F). For all $v_1, v_2\in \mathcal{B}_w$, $f\in\mathbb{F}$ and $g\in\mathbb{G}$ it holds:
\begin{itemize}
  \item[a)] $\lim_{n\to\infty}\frac1n \| U_{fg}^n v_1-U_{fg}^n v_2\|_{w} = 0$.
  \item[b)] $\lim_{n\to\infty}\frac1n \| U^n v_1-U^n v_2\|_{w} = 0$.
\end{itemize}
\end{lemma}

\begin{proof}
\begin{itemize}
  \item[a)]  First we claim for all  $v_1, v_2\in \mathcal{B}_w$, $f\in\mathbb{F}$, $g\in\mathbb{G}$ and $n\in\N$:
  $$ (U_{fg}^n v_1)(x) - (U_{fg}^n v_2)(x) \le \|v_1-v_2\|_{w} \Big( \beta \bar{\alpha}^n W(x) + (\beta K+1) \sum_{k=0}^{n-1} \bar{\alpha}^k\Big).$$
  Note that $\|v_1-v_2\|_{w}<\infty$ since $v_1, v_2\in \mathcal{B}_w$. The proof is by induction on $n$. For $n=1$ we obtain:
  \begin{eqnarray*}
       (U_{fg} v_1)(x)-(U_{fg} v_2)(x)  &\le&  \int (v_1(x') - v_2(x')) \psi_{v_1}(dx'|x,f(x),g(x)) \\
            &=&  \int \frac{v_1(x') - v_2(x')}{1+\beta W(x')} (1+\beta W(x'))  \psi_{v_1}(dx'|x,f(x),g(x)) \\
            &\le& \|v_1-v_2\|_{w} \big( 1+ \beta \int W(x') \psi_{v_1}(dx'|x,f(x),g(x))\big) \\
             &\le& \|v_1-v_2\|_{w} \big( \beta \bar{\alpha} W(x) +1+\beta K\big).
    \end{eqnarray*}
   This is the statement for $n=1$. Suppose the statement is true for $n$. For $n+1$ we obtain
   \begin{eqnarray*}
     &&  (U_{fg}^{n+1} v_1)(x)-(U_{fg}^{n+1} v_2)(x)  \le  \int \Big((U_{fg}^{n} v_1)(x')-(U_{fg}^{n} v_2)(x')\Big)  \psi_{U^n_{fg}v_1}(dx'|x,f(x),g(x)) \\
            &\le& \|v_1-v_2\|_{w}  \int \Big( \beta \bar{\alpha}^n W(x')+(\beta K+1) \sum_{k=0}^{n-1}\bar{\alpha}^k\Big)   \psi_{U^n_{fg}v_1}(dx'|x,f(x),g(x)) \\
            &\le& \|v_1-v_2\|_{w} \Big( \beta \bar{\alpha}^{n+1} W(x) + (\beta K+1) \sum_{k=0}^{n} \bar{\alpha}^k\Big).
    \end{eqnarray*}
   Interchanging the roles of $v_1$ and $v_2$ and dividing by $1+\beta W(x)$ yields that $$\| U_{fg}^n v_1-U_{fg}^n v_2\|_{w} \le \|v_1-v_2\|_{w} \cdot const.$$
  which implies the result.
  \item[b)] Similar to part a) and the derivation of the estimate in Theorem \ref{theo:Ucontract}.
\end{itemize}
\end{proof}

For the last result let us denote by $\mathcal{C}_w:=\{ v \in \mathcal{B}_w : v \mbox{ is continuous}\}$ and $\mathcal{C}_w^{(M)}:=\{ v \in \mathcal{B}_w^{(M)} : v \mbox{ is continuous}\}$. Moreover, we need the assumption

\begin{description}
  \item[(A4')] $W$ is continuous and for all $v\in \mathcal{C}_w$ the function  $(x,\mu,\nu) \mapsto \int e^{v(x')} \hat{Q}(dx'|x,\mu,\nu)$ is continuous.
\end{description}

\begin{remark}
Note that  $(A4')$ directly implies \eqref{eq.finite2}, if $\mathbb{E}_0$ is compact.
\end{remark}

Now we are able to prove our main result for the ergodic game.

\begin{theorem}\label{theo:arsr}
Let  $\gamma\in (-\gamma_0,\gamma_0)$ and assume (A1)-(A3),(A4'), (E) and (F).
\begin{itemize}
  \item[a)] The Poisson equation
            $$ \phi+v(x) = (Uv)(x),\quad x\in\mathbb{E}$$
            has a solution $(\phi^*,v^*)\in\R_+\times \mathcal{C}_w$ where $\phi^*$ is unique.
  \item[b)] There exist measurable functions $(f^*,g^*) \in  \mathbb{F}\times \mathbb{G}$ such that $(f^*(x),g^*(x))$ is a saddle point of
   \begin{eqnarray*} && (f,g)\mapsto  (U_{fg} v^*)(x),\quad x\in\mathbb{E}  \end{eqnarray*}
   for all $ (f,g) \in  \mathbb{F}\times\mathbb{G}$. Then $\phi^*$ is the value of the  average risk-sensitive game and the stationary policies $(f^*,f^*,\ldots)$ and $(g^*,g^*,\ldots)$ are average risk-sensitive optimal for player 1 and 2 respectively. In particular, $J_{f^*g^*}(x)=\phi^*=J(x)$ for all $x\in\mathbb{E}$.
  \end{itemize}
\end{theorem}

\begin{proof}
\begin{itemize}
  \item[a)]  We consider the following sequence of functions: $v_0 := 0, v_n := U^n 0$ for $n\in\N$. We claim that $v_n\in\mathcal{C}_w^{(M)}$. The statement is obvious for $v_0$. Suppose the statement is true for $v_n$. For $n+1$ we obtain:
      \begin{eqnarray*}
        \|v_{n+1}\|_{sp,w} &\le& \|v_{n+1}-v_n\|_{sp,w}+\ldots + \|v_{2}-v_1\|_{sp,w} +\|v_{1}\|_{sp,w} \\
         &=& \|U(U^{n}0)-U(U^{n-1}0)\|_{sp,w}+\ldots + \|U(U0)-U0\|_{sp,w} +\|U0\|_{sp,w} \\
         &\le& \alpha^{n-1}   \|U0\|_{sp,w} + \alpha^{n-2} \|U0\|_{sp,w}+\ldots +\|U0\|_{sp,w}\\
         &\le& \frac{\bar{r}}{1-\alpha}=M.
      \end{eqnarray*}
      Moreover, (A1)-(A3),(A4') imply that $v_{n+1}\in\mathcal{C}_w^{(M)}$.
 Hence according to Theorem \ref{theo:Ucontract} $(v_n)$ is a Cauchy sequence in $\mathcal{C}_w$ with limit $v^*\in \mathcal{C}_w$ and $\|Uv^*-v^*\|_{sp,w} =0$. Hence there exists a constant $\phi^*\in\R$ with $\phi^*=Uv^*-v^*$ or equivalently $\phi^*+v^*=Uv^*$. It remains to show the uniqueness of $\phi^*$. Suppose that there exists another pair $(\phi,v)\in\R\times\mathcal{C}_w $ which satisfies the Poisson equation, i.e. $\phi+v=Uv$. Iterating both inequalities $n$-times yields $n \phi+v = U^n v$ and $n\phi^*+v^*=U^nv^*$. Thus, $\frac1n \| n\phi+v-n\phi^*-v^*\|_{w} = \frac1n \|U^n v-U^nv^*\|_{w}\to0$ for $n\to\infty$ according to Lemma \ref{lem:v12}, which implies $\phi=\phi^*$.
  \item[b)] The existence of the saddle point $(f^*,g^*)$ follows again from the classical selection theorem and the minimax theorem. The saddle point property implies that $U_{f^*g^*} v^* = Uv^*$. Thus we obtain $n\phi^*+v^*= U_{f^*g^*}^n v^*$ which implies $\frac1n \|n\phi^*- U_{f^*g^*}^n v^*\|_{w} \to 0$ for $n\to\infty$. Since according to Lemma \ref{lem:v12} we obtain $\frac1n \|U_{f^*g^*}^n0- U_{f^*g^*}^n v^*\|_{w} \to 0$ for $n\to\infty$, it follows that $J_{f^*g^*}(x) =\phi^*\in\R$ for all $x\in\mathbb{E}.$  As in the proof of Theorem \ref{theo:posra} we obtain for arbitrary policies $\pi=(f_1,f_2,\ldots )$ and $\sigma=(g_1,g_2,\ldots )$ that
     $$ U_{f_1g^*}\ldots U_{f_ng^*}  v^*\le U_{f^*g^*}\ldots U_{f^*g^*}  v^*  \le U_{f^*g_1}\ldots U_{f^*g_n} v^*.$$ This implies again due to Lemma \ref{lem:v12} that $$J_{\pi g^*}\le J_{f^*g^*} \le J_{f^*\sigma}.$$ But this finally implies that the pair $\big((f^*,f^*,\ldots),(g^*,g^*,\ldots)\big)$ is a saddle point and $J(x)=J_{f^*g^*}(x)=\phi^*.$ Moreover, the last equation shows that $\phi^*\in\R_+$.
\end{itemize}

\end{proof}

\begin{remark}
Solutions of the Poisson equation can also be constructed by means of finite horizon games or by the so-called vanishing discount approach (see \cite{ri98}, for the risk-neutral case.)
\end{remark}

\begin{remark}\label{rem:SAT}
Assumption (E1) seems to be very restrictive at first glance because we suppose that it holds for all $v\in\mathcal{B}_w$. However, in applications it is often  reasonable to restrict $\mathcal{B}_w$ further to functions $v$ with certain properties. We give a quite general example here: Suppose that $\mathbb{E}=\R$ and $\mathbb{A}_x$ and $\mathbb{B}_x$ are independent of $x$. Moreover, let $\gamma<0$ and assume that $W$ is increasing (which is satisfied in almost all applications). Further we assume that a so-called 'monotone' risk-sensitive game is given, i.e. we assume
\begin{itemize}
  \item[(i)] $x\mapsto r(x,a,b)$ is increasing for all $a\in \mathbb{A}, b\in \mathbb{B}$.
  \item[(ii)] ${Q}(\cdot|x,a,b)$ is stochastically monotone for all $a\in \mathbb{A}, b\in \mathbb{B}$, i.e. ${Q}(\cdot|x,a,b) \le_{st} {Q}(\cdot|x',a,b)$ if $x\le x'$ where $\le_{st}$ is the usual stochastic order (see e.g. \cite{ms}, Chapter 5).
\end{itemize}
Under these structural assumptions on the game, it is easy to see that the $U$-operator maps increasing functions into increasing functions (for a discussion in the MDP context see \cite{br11}, chapter 2.4.4). In this case we can restrict our considerations to the set
$$ \mathcal{B}_w^I :=  \mathcal{B}_w\cap \{v:\mathbb{E}\to \R : v \mbox{ is increasing } \}$$
since the interesting functions $v_n$ in Theorem \ref{theo:arsr} satisfy $v_n=U^n0 \in \mathcal{B}_w^I$. Note that $v^*\in \mathcal{B}_w$ by the general theory and $v^*$ is increasing as the limit of the Cauchy sequence $(v_n)$. Hence we need (E1) only for functions $v\in\mathcal{B}_w^I$ in which case the condition reduces to:
\begin{eqnarray*}  \int W(x') \psi_v(dx'|x,\mu,\nu) &=& \frac{\int W(x') e^{\gamma v(x')} \hat{Q}(dx'|x,\mu,\nu)}{\int e^{\gamma v(x')} \hat{Q}(dx'|x,\mu,\nu)}\\
&\le & \int W(x')  \hat{Q}(dx'|x,\mu,\nu). \end{eqnarray*}
due to the fact that the random variables $W(X)$ and $ e^{\gamma v(X)}$ are negatively correlated which implies
$$ \int W(x') e^{\gamma v(x')} \hat{Q}(dx'|x,\mu,\nu) \le \Big( \int W(x') \hat{Q}(dx'|x,\mu,\nu)\Big)\Big(\int  e^{\gamma v(x')} \hat{Q}(dx'|x,\mu,\nu)\Big).$$
Hence (E1) is satisfied when
\begin{eqnarray}\label{eq:ass}  \int W(x') \hat{Q}(dx'|x,\mu,\nu) &\le&  \bar{\alpha} W(x) +K  \end{eqnarray}
which does not depend on $v\in\mathcal{B}_w$. A sufficient condition for \eqref{eq:ass} is
$$  \int W(x') {Q}(dx'|x,a,b) \le  e^{\gamma \bar{r}}\Big( \bar{\alpha} W(x) +K \Big)$$
for all $a\in \mathbb{A}, b\in \mathbb{B}$.
\end{remark}


{\bf Acknowledgements:} The authors would like to thank two referees for helpful comments and suggestions which improved the presentation of the paper.

\bibliographystyle{plainnat}

\end{document}